\documentclass[12pt,reqno]{amsart}
\usepackage{amsmath,amsfonts,epsfig,amssymb,graphics}

\newtheorem{thm}{Theorem}[section]
\newtheorem{cor}{Corollary}[section]
\newtheorem{prop}{Proposition}[section]
\newtheorem{lemma}{Lemma}[section]
\newtheorem{defn}{Definition}[section]
\newtheorem{rem}{Remark}[section]

\newcommand{\firstfunc}{\lambda_1(M) \text{Vol}(M)}
\newcommand{\kthfunc}{\lambda_k(M) \text{Vol}(M)}
\newcommand{\kthfunchigh}{\lambda_k(M) \text{Vol}(M)^{2/m}}
\newcommand{\firstfunchigh}{\lambda_1(M) \text{Vol}(M)^{2/m}}

\newcommand{\R}{\mathbb R}
\newcommand{\bbS}{\mathbb S}

\newcommand{\indice}{intersection index}

\begin{document}

\title[Spectrum of submanifolds]{Bounding the eigenvalues of the Laplace-Beltrami operator on compact submanifolds}
\author{Bruno Colbois}
\address{Universit\'e de Neuch\^atel, Institut de Math\'ematiques, Rue Emile-Argand 11, Case postale 158, 2009 Neuch\^atel
Switzerland}
\email{bruno.colbois@unine.ch}
\author{Emily B. Dryden}
\address{Department of Mathematics, Bucknell University, Lewisburg, PA 17837, USA}
\email{ed012@bucknell.edu}
\author{Ahmad El Soufi}
\address{Laboratoire de Math\'{e}matiques et Physique Th\'{e}orique, UMR-CNRS 6083, Universit\'{e} Fran\c{c}ois Rabelais de Tours, Parc de Grandmont, 37200 Tours, France}
\email{elsoufi@univ-tours.fr}

\thanks{The third author has benefitted from the support of the ANR (Agence Nationale de la Recherche) through FOG project ANR-07-BLAN-0251-01.}

\date{}
\begin{abstract} We give upper bounds for the eigenvalues of the La-place-Beltrami operator of a compact $m$-dimensional submanifold $M$
 of $\R^{m+p}$. Besides the dimension and the volume of the submanifold and the order of the eigenvalue, these bounds depend on either the maximal number of intersection points of $M$ with a $p$-plane in a generic position (transverse to $M$), or an invariant which measures the concentration
of the volume of $M$ in $\R^{m+p}$. These bounds are asymptotically optimal in the sense of the Weyl law. On the other hand, we show that even for hypersurfaces (i.e., when $p=1$), the first positive eigenvalue cannot be controlled only in terms of the volume, the dimension and (for $m\ge 3$) the differential structure. 

\end{abstract}




\maketitle

\section{Introduction}

Let $M$ be a compact, connected submanifold without boundary of dimension $m \geq 2$ immersed in a Euclidean space $\R^{m+p}$ with $p\ge1$; that is, $M$ is the image of a compact smooth manifold $\bar M$ of dimension $m$ by an immersion $X:\bar M\to\R^{m+p}$ of class $C^2$. 
We denote by  $g$  the Riemannian metric naturally induced on $\bar M$ (that is, the first fundamental form of the submanifold $M$) and by $\Delta$ the corresponding Laplace-Beltrami operator whose spectrum consists in an unbounded sequence of eigenvalues
$$
\text{Spec}(\Delta) = \{ 0 = \lambda_0(M) < \lambda_1(M) \leq \lambda_2(M) \leq \cdots \leq \lambda_k(M) \leq \cdots \}.
$$
Consider the $k$-th eigenvalue as a functional
$$
M \mapsto \lambda_k (M)$$
on the space of all immersed $m$-dimensional submanifolds of fixed volume of $\R^{m+p}$; alternatively, consider the normalized dilation-invariant functional 
$$
M \mapsto  \lambda_k (M) \text{Vol}(M)^{2/m}
$$
on the space of all immersed $m$-dimensional submanifolds of $\R^{m+p}$.  Of course, these two functionals have the same variational properties.

Hersch \cite{H} was the first to obtain a result on these functionals.  Indeed, let $\bbS^2$ be the standard $2$-sphere naturally embedded in $\R^3$. Hersch proved that if $M$ is any 2-dimensional immersed surface of genus zero of $\R^{2+p}$ (that is, $M=X(\bbS^2)$, where $X:\bbS^2\to\R^{2+p}$ is an immersion), then  
$$
\firstfunc \leq \lambda_1(\bbS^2)\text{Vol}(\bbS^2)=8 \pi.
$$
Moreover, the equality holds if and only if $M$ has constant Gaussian curvature.

One decade later, Yang and  Yau \cite{YY}  proved  that if $M$ is  
an orientable immersed surface of genus $\gamma$ (that is, $M=X(\bar M)$ where $\bar M$ is an orientable compact surface of genus $\gamma$ and $X:\bar M\to\R^{2+p}$ is an immersion), then  $\firstfunc \leq 8 \pi (\gamma +1)$.  This bound has been improved by Ilias and the third author   \cite{EI1} as follows:
$$
\firstfunc \leq 8 \pi \left\lfloor \frac{\gamma + 3}{2}\right\rfloor ,
$$
where $\lfloor \cdot \rfloor$ denotes the floor function.  

A similar upper bound was obtained in the nonorientable case by Li and Yau \cite{LY}. However, these upper bounds are not optimal in general and the exact value of the supremum of $\firstfunc $ among surfaces of fixed topology is known only in the following cases: immersed spheres (Hersch \cite{H}), tori  (Nadirashvili \cite{N}), projective planes (Li and Yau \cite{LY}) and  Klein bottles (\cite{JNP} and \cite{EGJ}).  A conjecture concerning orientable surfaces of genus 2 is stated in \cite{JLNNP}.

The extension of the result of Yang and Yau to higher order eigenvalues was obtained by Korevaar in \cite{K}:  
there exists a universal constant $C>0$ such that for any integer $k \geq 1$ and any compact orientable surface $M$ of genus $\gamma$, we have
$$
\kthfunc \leq C(\gamma + 1) k.
$$
Note that the estimates given above for the first nonzero eigenvalue $\lambda_1$ are based on ``barycentric type methods,'' i.e., the use of coordinate functions as test functions after a suitable transformation that puts the barycenter at the origin; for a typical example of this classical method, see \cite{EI1} or \cite{H}.  These barycentric methods do not apply to higher order eigenvalues, and Korevaar's proof required new techniques.

In dimensions three and higher, the situation differs significantly from the 2-dimensional case.  
Indeed, it follows from the Nash embedding theorem and the result of Dodziuk and  the first author  \cite{CD} that in dimension $m \geq 3$, we have 
$$
\sup_M \firstfunchigh  = \infty,
$$
where the supremum is taken over all compact submanifolds $M$ with fixed smooth structure (that is, $M=X(\bar M)$, where $\bar M$ is a fixed compact smooth manifold of dimension $m\ge3$ and $X:M\to\R^{m+p}$ is a smooth immersion from $M$ into $\R^{m+p}$ for some $p\ge 1$).  

To study extremal properties of the spectrum, it is therefore necessary to 
impose additional constraints, either of an intrinsic or an extrinsic nature.  For example, we can assume that the induced metric $g$ preserves a conformal class of metrics \cite{CE, EI, K}, 
a symplectic or a K\"{a}hler structure \cite{P, BLY}, the action of a Lie group \cite{AF,CDE, E1}, etc.  Regarding results with constraints of extrinsic type, a well-known example is given by Reilly's inequality \cite{R,EI2}: 
 $$ \lambda_1(M)\le\frac m {\text{Vol}(M)} \|H(M)\|_2^2,$$
 where $\|H(M)\|_2$ is the $L^2$-norm of the mean curvature vector field of $M$. More generally, it follows from results of Harrell, Ilias and the third author \cite{EHI} and the recursion formula of Cheng and Yang \cite[Corollary 2.1]{CY} that for any positive integer $k$,
  $$ \lambda_k(M)\le R(m) \|H(M)\|_\infty ^2 \ k^{2/m},$$
where $\|H(M)\|_\infty$ is the $L^\infty$-norm of $H(M)$ and $R(m)$ is a constant depending only on $m$.
 
In this paper, we will focus on other extrinsic constraints. The
first one is related to the following invariant: For a compact immersed submanifold $M$ of dimension $m$ in $\R^{m+p}$, almost all the $p$-planes $\Pi$ in $\R^{m+p}$  are transverse to $M$
so that the intersection $\Pi \cap M$ consists of a finite number of points. We define
the \emph{\indice \ of $M$}  as the supremum
$$i(M)=\sup_\Pi \#  M\cap \Pi ,$$ where $\Pi$ runs over the set of all $p$-planes which are transverse to $M$ in $\R^{m+p}$; if $M$ is not embedded, we count multiple points of $M$ according to their multiplicity. For instance, the \indice \ of a hypersurface $M$ is the ``maximal'' number of collinear points in $M$. Note that the invariant $i(M)$ is related to the notion of \emph{width} introduced by Hass, Rubinstein and Thompson \cite[Definition 1]{HRT}. 

We begin by obtaining an estimate involving the \indice \ for the lowest positive eigenvalue $\lambda_1$.  Note that $B^m$ denotes the Euclidean ball of radius $1$ in $\R^m$ and
$\mathbb{S}^m$ denotes the unit sphere in $\R^{m+1}$.

\begin{thm}\label{betterestimate}
For every compact $m$-dimensional immersed submanifold $M$ of a Euclidean space $\R^{m+p}$ we have
\[
\firstfunchigh \leq   A(m)  \left(\frac{i(M)}{2}\right)^{1+\frac{2}{m}} \lambda_1(\mathbb{S}^m)\text{Vol}(\mathbb{S}^m)^{2/m} ,
\]
where $\lambda_1(\mathbb{S}^m)=m$ and $A(m)=\frac{m+2}2\frac{\text{Vol}(\mathbb{S}^m)}{\text{Vol}(\mathbb{S}^{m-1})}=\frac{m+2}2{\sqrt\pi}\frac{\Gamma(\frac{m} 2)}{\Gamma(\frac {m+1} 2)}$.  
\end{thm}
\noindent The proof of this theorem again uses the barycenter method.

The result of Theorem \ref{betterestimate} extends to higher order eigenvalues but with a less explicit upper bound. 
\begin{thm}\label{mainth}
For every compact $m$-dimensional immersed submanifold $M$ of $\R^{m+p}$ and every positive integer $k$, we have
\[
\kthfunchigh \leq   c(m) i(M)^{2/m} k^{2/m},
\]
where $c(m)$ is a constant depending only on the dimension $m$ of $M$ which is given explicitly by \eqref{eqn:bigC} and \eqref{eqn:littlec}.
\end{thm}
The exponent of $k$ in the estimate is asymptotically best possible as follows from the Weyl law.  The estimate itself is not sharp, because the constants are not optimal.

For a convex hypersurface we clearly have $i(M)=2$. Thus, 
\begin{cor}
If $M$ is a compact convex hypersurface of $\R^{m+1}$, then
\[
\firstfunchigh \leq  A(m) \ \lambda_1(\mathbb{S}^m)\text{Vol}(\mathbb{S}^m)^{2/m}
\]
and, for all $k\ge 2$,
\[
\kthfunchigh \leq  c(m) \ 2^{2/m}\ k^{2/m}.
\]
\end{cor}

When $M$ is a real algebraic hypersurface defined by a polynomial equation of degree $N$, it is easy to show that $i(M)\le N$. This implies

\begin{cor} Let $P$ be a real polynomial in $m+1$ variables and of degree $N$ such that  $M=P^{-1}(0)\subset \R^{m+1}$ is a compact hypersurface. Then, for all $k\ge 1$, 
\[
\kthfunchigh \leq   c(m) N^{2/m}k^{2/m} .
\]
\end{cor}
\noindent A more general result, in arbitrary codimension, will be given in Corollary  \ref{coro1}.

These results tell us that to the extent that the \indice \ of $M$ is controlled, the spectrum of $M$ is also controlled. 
In fact, Theorem \ref{mainth} can be understood as a consequence of a more abstract result given in Theorem \ref{concentration} below.  It turns out that control of the intersection index $i(M)$ suffices to give control on the ``concentration of the volume'' of $M$ in $\R^{m+p}$ in the sense of the following definition. 

\begin{defn}\label{def} Let $m\ge 2$ and $p\ge1$ be two integers and let $L$ be a positive real number. We denote by  
$\mathcal M(m,p,L)$ the class of all $m$-dimensional compact immersed submanifolds of $\R^{m+p}$ such that, for all $x\in \R^{m+p}$ and all $r>0$, 
$$\text{Vol}_m(B(x,r)\cap M)\le L \ r^m,$$
 where
$B(x,r)$ denotes the Euclidean ball of center $x$ and radius $r>0$ in $\R^{m+p}$. 
\end{defn}

The eigenvalues of submanifolds in $\mathcal M(m,p,L)$ are uniformly controlled. Indeed, one has the following

\begin{thm} \label{concentration} Let $m\ge 2$ and $p\ge1$  be two integers and let $L$ be a positive real number.  
For all $M\in
\mathcal M (m,p,L)$ and all $k\ge 1$, we have
$$\kthfunchigh\le C(m) L^{2/m} k^{2/m},$$
where $C(m)$ is a constant depending only on $m$ given explicitly in \eqref{eqn:bigC}. 
\end{thm}

The result of Dodziuk and the first author \cite{CD} together with Theorem  \ref{betterestimate}, Theorem  \ref{concentration} and Reilly's inequality tell us that, given a smooth manifold $\bar M$ of dimension $m\ge 3$, there  exist Riemannian metrics $g$ of volume one on $\bar M$ such that any immersion of $\bar M$ into a Euclidean space $\R^{m+p}$ which preserves $g$ must have a very large \indice, very large total mean curvature, and volume which concentrates into a small Euclidean ball.  More precisely,

\begin{cor} Let $\bar M$ be a compact smooth manifold of dimension $m\ge 3$. For every integer $K>0$, there exists a Riemannian metric $g_K$ of volume one on $\bar M$ such that, for \textbf{any} isometric immersion $X$ from $(\bar M,g_K)$ into a Euclidean space $\R^{m+p}$ (with arbitrary $p$), the submanifold $M= X(\bar M)$ satisfies the following conditions:
\begin{enumerate}
	\item there exists a $p$-plane $\Pi\subset \R^{m+p}$ transverse to $M$ which intersects $M$ at least $K$ times.
	\item $\|H( M)\|_2 >K  $
	\item there exists a Euclidean ball $ B(x,r)\subset\R^{m+p}$ such that the volume of the portion of $M$ lying in $ B(x,r)$ is larger than the volume of $K$ Euclidean balls of radius $r$ and dimension $m$, that is $$1\ge \text{Vol}_m(B(x,r)\cap  M)>  K \text{Vol}_m(B^m) r^m.$$
	
\end{enumerate}

\end{cor} 

In the three theorems above, there is no restriction on the codimension of the submanifold.  Therefore, our bounds on the eigenvalues depend on purely extrinsic assumptions, in the sense that all differential structures and  Riemannian metrics are allowed. An interesting question is to know what happens if the codimension is assumed to be 1. \\

\noindent\emph{Is the first eigenvalue bounded upon the set of all compact hypersurfaces of $\R^{m+1}$ of fixed volume and (for $m\ge 3$) fixed smooth structure}?\\

We conclude by providing a negative answer to this question. Indeed, putting together results from the literature, we prove the following:
\begin{thm}\label{large}
\begin{enumerate}
\item There exists a sequence of compact orientable surfaces $M_n$ embedded in $\R^3$ such that 
$$\lambda_1(M_n) \text{Vol}(M_n) {\longrightarrow} \infty$$
as $n \to \infty$.
\item Let $p_0$ be a positive integer and let $M$ be any compact smooth submanifold of dimension $m\ge 3$ of $\R^{m+p_0}$. Then there exists a sequence $M_n$ of smooth submanifolds embedded in $\R^{m+p_0}$ all diffeomorphic to $M$, such that 
$${\lambda_1(M_n) \text{Vol}(M_n)^{2/m}}{\longrightarrow} \infty$$
as $n \to \infty$.
\end{enumerate}
\end{thm} 
In other words, for any fixed compact submanifold $M$ of dimension $m\ge 3$ of $\R^{m+p_0}$,
$$\sup_X\lambda_1(X(M)) \text{Vol}(X(M))^{2/m}=\infty$$
where the supremum is taken over all embeddings from $M$ into  $\R^{m+p_0}$.

The structure of the paper is as follows. In section \ref{sec:enoughvol} we study the volume of the projection of an $m$-dimensional submanifold onto an $m$-plane and explain how the control on the \indice \ implies a control on the volume concentration. In section \ref{sec:better} we prove
 Theorem \ref{betterestimate} using the barycenter method and estimates from section \ref{sec:enoughvol}. In section \ref{sec:concentration} we give the proofs of Theorem \ref{concentration} and Theorem \ref{mainth}. In the last section we explain how we get Theorem \ref{large}.


\section{Volume of orthogonal projections onto $m$-planes  \label{sec:enoughvol}}

Let $M$ be an immersed submanifold of dimension $m \geq 2$ in $\R^{m+p}$. In this section, we do not need to assume that $M$ is compact but only that $M$ has finite volume and finite \indice \ $i(M)$. 

We begin by proving that there exists an $m$-dimensional linear subspace $H \subset \R^{m+p}$ such that the volume of the image of $M$ under the orthogonal projection on $H$ is bounded below in terms of $\text{Vol}(M)$ and $i(M)$. 

Let $G:=G(m,m+p)$ be the Grassmannian of $m$-planes through the origin in $\R^{m+p}$. 
To each $m$-plane $H$ in $G$  we associate  
the orthogonal projection $\pi_H:M \subset\R^{m+p} \rightarrow H$. From the definition of $i(M)$ it is immediate that, generically, at most $i(M)$ points on $M$ have the same image under $\pi_H$. The volumes of $\pi_H(M)$ 
and of $M$ are related as in the following lemma.

\begin{lemma} We endow $G$ with its O(n)-invariant Radon measure of total volume one (see, e.g., \cite{Mat}).
Then, 
\[
\int_{G} \text{Vol} (\pi_H (M)) dH \geq \frac{2}{i(M)} \frac{\text{Vol}(B^m)}{\text{Vol}(\mathbb{S}^m)} \text{Vol}(M),
\]
with $ \frac{\text{Vol}(B^m)}{\text{Vol}(\mathbb{S}^m)}=\frac{1}{m \sqrt{\pi}}\frac{\Gamma(\frac {m+1}{2})}{\Gamma(\frac{m}{2})}$.
\end{lemma}

\begin{proof}  Let $H\in G$ be an $m$-plane through the origin endowed with its standard volume element $v_H$.  Let $\theta_H$ be the function on $M$ defined by $$\pi_H^* v_H = \theta_H v_M,$$
 where $v_M$ is the Riemannian volume element of $M$. If $\xi_1, \ldots, \xi_m$ constitute an orthonormal basis of the tangent space to $M$ at a point $x$, then 
\[
\theta_H(x) = \pm\det (\pi_H(\xi_1(x)), \ldots, \pi_H(\xi_m(x))) ,
\]
where the determinant is taken with respect to an orthonormal basis of $H$.

Since a generic point in $ \pi_H(M)$ has at most $i(M)$ preimages under the map $\pi_H:M \rightarrow \pi_H(M) \subset H$, one easily checks that   
\[
\int_M \vert\pi_H^* v_H\vert= \int_M \vert\theta_H(x)\vert v_M\leq i(M) \int_{\pi_H(M)} v_H = i(M) \text{Vol}(\pi_H(M)).
\]
Integrating over $G$ we get 
\begin{eqnarray}\label{eqn:vol}
i(M) \int_{G} \text{Vol}(\pi_H(M))dH & \geq &  \int_{G}dH \int_M \vert\theta_H(x)\vert v_M \\
\nonumber & = & \int_M \left( \int_{G}\vert\theta_H(x)\vert dH\right) v_M.
\end{eqnarray}
Now, from the definition of $\theta_H$ and the invariance of the measure of $G$, we easily deduce that the integral $I(G)=\int_{G}\vert\theta_H(x)\vert dH$ does not depend on the point $x$.
Indeed, if $\rho\in O(n)$ is such that $\rho\cdot T_yM = T_xM$, then for all $H\in G$ and $V\in T_yM$,  $\rho\cdot\pi_H (V) = \pi_{\rho\cdot H}(\rho\cdot V)$ and $\vert\theta_{H}(y)\vert= \vert\theta_{\rho\cdot H}(x)\vert$.
Hence,
\begin{eqnarray}\label{I(G)}
i(M) \int_{G} \text{Vol}(\pi_H (M))dH  \geq  I(G)\text{Vol}(M).
\end{eqnarray}
To determine the value of $I(G)$, we consider the case of the sphere $\mathbb{S}^m \subset \R^{m+p}$.  In this case, $i(\mathbb{S}^m) = 2$ and, for all $H \in G$,
$$\int_{\mathbb{S}^m}\vert\theta_H(x)\vert v_{\mathbb{S}^m}= \int_{\mathbb{S}^m} \vert\pi_H^* v_H\vert =2\text{Vol}(\pi_H(\mathbb{S}^m))=2\text{Vol}( B^m). $$
Hence, 
 $$ I(G)\text{Vol}({\mathbb{S}^m})=\int_{G}dH \int_{\mathbb{S}^m}\vert\theta_H(x)\vert v_{\mathbb{S}^m} =2 
 \text{Vol}( B^m).$$
Substituting into (\ref{I(G)}) we obtain the desired inequality.
\end{proof}
\begin{rem}
One can prove that the preceding lemma holds for more general functions than projections.  
However, this makes both the statement and the proof of the result more complicated, 
and we will not need this full generality in what follows.
\end{rem}

\begin{cor}\label{cor1}
There exists an $m$-plane $H \in G$ such that 
\[
\text{Vol}(\pi_H(M)) \geq  \frac{2}{i(M)}\frac{\text{Vol}(B^m)}{\text{Vol}(\mathbb{S}^m)}\text{Vol}(M).
\]
\end{cor}  

\medskip
It is possible to apply these considerations to the intersection of $M$ with a Euclidean ball $B(x,r)$  of center $x$ and radius $r>0$ in $\R^{m+p}$. Therefore, there exists an $m$-plane $H \in G$ (depending on $x$ and $r$) such that 
\[
\text{Vol}(\pi_H(M \cap B(x,r))) \geq  \frac{2}{i(M)}\frac{\text{Vol}(B^m)}{\text{Vol}(\mathbb{S}^m)}\text{Vol}(M\cap B(x,r)).
\]

On the other hand,  $\pi_H(M \cap B(x,r))$ is contained in $\pi_H( B(x,r))$, which is an $m$-dimensional Euclidean ball of radius $r$ in the $m$-plane $H$; thus
$$\text{Vol}(\pi_H(M \cap B(x,r))) \le r^m \text{Vol}(B^m).$$
Hence, we have also proved the following

\begin{prop} \label{fundprop} For all $x \in \R^{m+p}$ and all $r>0$, we have
\[
 \text{Vol}(M \cap B(x,r)) \le \frac{i(M)}{2}  \text{Vol}(\mathbb{S}^m)\ r^m .
\]
\end{prop}

Roughly speaking, the control of $i(M)$ ensures that $M$ cannot concentrate
in small parts of $\R^{m+p}$. 


\section{Estimating $\lambda_1$ : proof of Theorem  \ref{betterestimate}} \label{sec:better} In this section, we prove the estimate of Theorem \ref{betterestimate} 
using the barycenter method.
Recall that $M=X(\bar M)$ is an $m$-dimensional compact  immersed submanifold in $\R^{m+p}$. 
We denote by $d_{m+p}$ and $d_m$ the usual ``distance to origin'' functions
in $\R^{m+p}$ and $\R^m$, respectively.

\smallskip
Up to a displacement, we can assume that the center of mass of $M$ is at the origin so that,
for $1\le i\le m+p$,
$$\int_Mx_iv_M=0$$
where $v_M$ denotes the induced volume element on $M$.
Hence, for each $1\le i\le m+p$, we have
\begin{equation}\label{eqn:ithRay}
\lambda_1(M)\int_Mx_i^2v_M \le \int_M\vert \nabla x_i\vert^2v_M.
\end{equation}

Let $\{\xi_k\}_{k=1}^m$ be an orthonormal basis for the tangent space to $M$ at a point $x$, and let $\{e_i\}_{i=1}^{m+p}$ be the standard orthonormal basis of $\mathbb{R}^{m+p}$.  Then we have  
\begin{eqnarray*}
\sum_{i=1}^{m+p}\vert \nabla x_i\vert^2 & = & \sum_{i=1}^{m+p} \sum_{k=1}^m < \nabla x_i, \xi_k>^2 \\
& = & \sum_{k=1}^m \sum_{i=1}^{m+p} < e_i, \xi_k >^2 = m.
\end{eqnarray*}
Summation in \eqref{eqn:ithRay} thus gives
\begin{equation}\label{eqlambda}
\lambda_1(M)\int_M\vert x \vert^2v_M \le m \text{Vol}(M).
\end{equation}

The key now is to obtain a lower bound for the integral $\int_M\vert x \vert^2v_M$, often called ``moment of inertia," which may also be written $\int_M d^2_{m+p}(x)v_M$. To this aim, we will use
Corollary \ref{cor1} and the following  

\begin{lemma} \label{l1} Let $\Omega$ be a domain in $\R^m$. Then 

$$\int_{\Omega} d_m^2(x)dx \ge \int_{\Omega^*}d_m^2(x)dx,$$
where $\Omega^*$ denotes the Euclidean ball in $\R^m$ centered at the origin and with the same volume as $\Omega$.
\end{lemma}

A short proof of this classical result (see, e.g., \cite[p.153]{PS}), kindly communicated by Asmaa Hasannezhad, 
goes as follows:
if  $\Omega^*$ is a ball of radius $R$ we have  
$\Omega=(\Omega \cap \Omega^*)\cup(\Omega \setminus \Omega^*),$ and
\begin{eqnarray*}
\int_{\Omega }d_m^2(x)dx & =& \int_{\Omega \cap \Omega^*}d_m^2(x)dx + \int_{\Omega \setminus \Omega^*}d_m^2(x)dx \\
& \geq & \int_{\Omega \cap \Omega^*}d_m^2(x)dx + \int_{\Omega \setminus \Omega^*}R^2dx \\
& \geq & \int_{\Omega^* }d_m^2(x)dx,
\end{eqnarray*}
since $\text{Vol}(\Omega \setminus \Omega^*)=\text{Vol}(\Omega^*\setminus \Omega)$.
\begin{proof}[Proof of Theorem \ref{betterestimate}] From
Corollary \ref{cor1}, there exists an $m$-plane $H \in G$  such that 
\begin{equation}\label{eqH}
\text{Vol}(\pi_H(M)) \geq  \frac{2}{i(M)}\frac{\text{Vol}(B^m)}{\text{Vol}(\mathbb{S}^m)}\text{Vol}(M).
\end{equation}
The idea is to compare the situations on $M$ and on $\Omega=\pi_H(M)$. Since the projection $\pi_H$ is length-decreasing, we have
\[
\int_M d^2_{m+p}(x)v_M \ge \int_Md_m^2(\pi_H(x))v_M.
\]
Note that there is a substantial loss of information when we perform the projection $\pi_H$; in particular, this inequality can be seen to be strict even in the case $M = \mathbb{S}^m$.

On the other hand, since $M$ is compact, almost every point of $\Omega=\pi_H(M)$ admits at least two preimages in $M$.  Thus 
\[
2\int_{\Omega}d_m^2(x)v_H \le \int_Md_m^2(\pi_H(x))\vert\pi_H^*v_H \vert\le \int_M d_m^2(\pi_H(x))v_M,
\]
where the inequality on the right is due to the fact that $\pi_H$ is volume decreasing (with the notation of \S \ref{sec:enoughvol},  $\pi_H^*v_H=\theta_Hv_M$ with $\vert\theta_H\vert \le 1$). Consequently, 
\[
\int_M d^2_{m+p}(x)v_M \ge 2\int_{\Omega}d_m^2(x)dx .
\]
Applying Lemma \ref{l1} we get, through obvious identification of $H$ with $\R^m$,
\begin{equation}\label{eqd}
\int_M d^2_{m+p}(x)v_M \ge 2\int_{\Omega}d_m^2(x)dx \ge 2\int_{\Omega^*}d_m^2(x)dx,
\end{equation}
where $\Omega^*$ is the ball in $H$ centered at the origin with  $\text{Vol}(\Omega^*)=\text{Vol}(\Omega)$. If we denote by $\rho$ the radius of the ball $\Omega^*$, then we have 
\[
\int_{\Omega^*}d_m^2(x)dx= \rho^{m+2}\int_{B^m}\vert x\vert ^2dx
\]
and $$\text{Vol}(\Omega^*)=\rho^m \text{\text{Vol}}(B^m).$$
Hence, 
$$\rho = \left(\frac{\text{Vol}(\Omega^*)}{\text{Vol}(B^m)}\right)^{\frac{1}{m}}=\left(\frac{\text{Vol}(\pi_H(M))}{\text{Vol}(B^m)}\right)^{\frac{1}{m}}$$
and 
\[
\int_{\Omega^*}d_m^2(x)v_H =\left(\frac{\text{Vol}(\pi_H(M))}{\text{Vol}(B^m)}\right)^{1+\frac{2}{m}} \int_{B^m}\vert x\vert^2dx.
\]
Combining this last equality with (\ref{eqd}) and  (\ref{eqH}) above, we get the following lower bound for $\int_M d^2_{m+p}(x)v_M$ :
\begin{equation}\label{eqint}
\int_M d^2_{m+p}(x)v_M \ge 2 \left(\frac{2\text{Vol}(M)}{i(M)\text{Vol}(\mathbb{S}^m)}\right)^{1+\frac{2}{m}} \int_{B^m}\vert x\vert^2dx.
\end{equation}
When we put this estimate into (\ref{eqlambda}) we get 
$$2\lambda_1(M)\left(\frac{2\text{Vol}(M)}{i(M)\text{Vol}(\mathbb{S}^m)}\right)^{1+\frac{2}{m}} \int_{B^m}\vert x\vert^2dx \le m \text{Vol}(M);$$
rearranging terms gives
\[
\lambda_1(M)\text{Vol}(M)^{2/m} \le 
\frac{\text{Vol}(\mathbb{S}^m)}{2\int_{B^m} \vert x\vert^2dx} m \text{Vol}(\mathbb{S}^m)^{2/m} \left(\frac{i(M)}{2}\right)^{1+\frac{2}{m}},
 \]
which completes the proof of Theorem \ref{betterestimate} since 
\[
\int_{B^m} \vert x\vert^2dx=\frac 1{m+2}\text{Vol}(\mathbb{S}^{m-1}).
\]
\end{proof}
 

\section{Proofs of Theorem \ref{mainth} and Theorem \ref{concentration} \label{sec:concentration}}

A classical way to construct test functions for the Rayleigh quotient 
on a Riemannian manifold is to consider a family of $k$ mutually disjoint geodesic balls of radius $2r$, and a corresponding family
of cut-off functions.  Each cut-off function takes the value $1$ on one geodesic ball of radius $r$ and $0$ outside the corresponding
ball of radius $2r$. To estimate the Rayleigh quotient we need some information about the local
geometry of $M$, in general in terms of a lower bound on the Ricci curvature so that Bishop-Gromov volume comparison holds. This method is not convenient for the eigenvalues of the Neumann Laplacian on a domain or for the eigenvalues of a compact submanifold. In \cite{CM}, Maerten and the first author introduced a more elaborate family of sets. The construction that they propose is a metric one and can be adapted to 
various situations.  We explain the method in the case of submanifolds.

Let $M$ be a compact, connected immersed submanifold of dimension $m \geq 2$ in $\R^{m+p}$ and let $L >0$ be such that $M\in
\mathcal M (m,p,L)$. We endow $\R^{m+p}$ with the Euclidean distance $d$ and the Borel measure $\mu$ with support in $M$ defined by: 
$$\mu(A):= \text{Vol}_m(M\cap A).$$
The hypotheses (H1) and (H2) of section 2  of \cite{CM} are satisfied. 
In the present situation, (H1) states that we can cover a ball in $\R^{m+p}$ of radius $4r$ by a bounded number of balls of radius $r$, while (H2) states that the measure of the $r$-balls tends uniformly to zero with the radius $r$.
Indeed, in $\R^{m+p}$, the number $C(r)$ of balls of radius $r$ we need to cover a ball of radius $4r$ is controlled independently of $r$; for example, we can take $C(r)\le 8^{m+p}$  (see  Example 2.1 of \cite{CM} for a more general estimate). On the other hand, the fact that $M$ is in $\mathcal M(m,p,L)$ implies that
 for each $x \in \R^{m+p}$  and $r>0$, 
$$\mu(B(x,r)) \le L r^m.$$
Thus $\mu(B(x,r))$ tends to zero with $r$ uniformly with respect to $x$.
Therefore,   Corollary 2.3 of \cite{CM} enables us to state the following
\begin{prop}\label{resultCM}
 Let $K$ be a positive integer and  $\alpha$  a positive real number with $\alpha \leq \frac{\omega}{2\cdot 8^{m+p}K}$, where $\omega = \mu(\R^{m+p}) = \text{Vol}(M)$. Let $r>0$ be 
such that 
$$2\ \sup\left\{ 8^{m+p}\mu(B(x,r))\; ; \; x \in \R^{m+p}\right\} \leq \alpha.$$
Then there exist $K$ measurable subsets $A_1, \ldots, A_K \subset \R^{m+p}$ such that $\mu(A_i) \geq \alpha$ 
and, for each $i \neq j$, $d(A_i,A_j) \geq 3r$. 
\end{prop}

\begin{proof}[Proof of Theorem \ref{concentration}]
Our ultimate goal is to construct $k+1$ disjointly supported test functions on $M$ whose Rayleigh quotients are controlled in terms of $L$ and $k$.  First, let $K=2k+1$ and let $A_1, \ldots, A_K$  be $K$  disjoint measurable subsets in $\R^{m+p}$ satisfying Proposition \ref{resultCM} 
with  $\alpha = \frac{\omega}{6\cdot 8^{m+p}  k} = \frac{\text{Vol}(M)}{6\cdot 8^{m+p} k}$.  Denote by $A_i^r=\{x \in \R^{m+p} \; ;\; d(x,A_i)<r\}$ 
the $r$-neighborhood of $A_i$.  
A priori, we have no control over the volume of the portion of $M$ contained in $A_i^r$, so we will make a choice of $k+1$ sets amongst our disjoint $2k+1$ measurable subsets.  Namely, since $d(A_i,A_j) \geq 3r$ for $i \neq j$, the $A_i^r$ are mutually disjoint and it is clear that the number
\[
Q = \# \left\{i \in 1, \ldots, 2k+1 \; ;\; \mu(A_i^r) \geq \frac{\text{Vol}(M)}{k}\right\} 
\]       
is less than $ k$. Therefore, there exist at least $k+1$ subsets, say $A_1,\ldots, A_{k+1}$, amongst $A_1, \ldots, A_K$  with the property that $\mu(A_i^r) < \frac{\text{Vol}(M)}{k}$.  From the definition of $\mu$, the  $k+1$ disjoint measurable sets $A_1^r \cap M,\ldots, A_{k+1}^r \cap M$ also satisfy this estimate.  By abuse of notation, we will refer to these sets on $M$ as $A_i^r$ as well.

As in \cite{CM}, we construct a family of test functions $\varphi_1,\ldots\varphi_{k+1}$ on $M$ as follows: for $i\le k+1$,  $\varphi_i$ is equal to $1$ on $A_i$, vanishes outside the $r$-neighborhood $A_i^r$ of $A_i$ and $\varphi_i(x)=1-\frac{d(x,A_i)}r$ on $A_i^r\setminus A_i$.  Observing that $|\nabla \varphi_i(x)|\le \frac 1 r$ almost everywhere in $A_i^r\setminus A_i$, a straightforward calculation shows that the Rayleigh quotient of $\varphi_i$ is given by 
\[
R(\varphi_i) \le \frac{1}{r^2}\frac{\mu(A_i^r)}{\mu(A_i)} < \frac{1}{r^2} \frac{\frac{\text{Vol}(M)}{k}}{\alpha} = \frac{6 }{r^2}8^{m+p}.
\]
Recall that the number $r$ must be chosen such that 
\[
2\cdot 8^{m+p} \mu(B(x,r)) \leq \frac{\text{Vol}(M)}{6k8^{m+p}}.  
\]
But, since  $M\in\mathcal M(m,p,L)$,
\[
 \mu(B(x,r))\le   L r^m 
\]
and we can take 
\[
r = \left(\frac{\text{Vol}(M)}{L k}\frac{1}{12\cdot 8^{2(m+p)}}\right)^{1/m}.
\] 
Therefore, the estimate of the Rayleigh quotient above 
becomes
\[
R(\varphi_i) \leq \left(\frac{k}{\text{Vol}(M)}\right)^{2/m} L^{2/m}C(m,p)
\]
with $C(m,p)= 6\cdot 8^{m+p}(12\cdot 8^{2(m+p)})^{2/m}$.

Invoking the min-max principle, one deduces the estimate  
$$\lambda_k(M)\le \max_{i\le k+1}  R(\varphi_i) \leq\left(\frac{k}{\text{Vol}(M)}\right)^{2/m} L^{2/m}C(m,p).$$
Since $\lambda_k(M)$ is an intrinsic invariant, the Nash embedding theorem \cite{Nash} says that one can assume without loss of generality that the codimension $p$ is uniformly bounded above in terms of $m$. Hence the constant $C(m,p)$, which is increasing in $p$, can be replaced  by a constant $ C(m)$ which depends only on the dimension $m$.  
In particular, we may take \cite{Nash} $p=2m^2+5m$ to get
\begin{equation}\label{eqn:bigC}
C(m) = 6 \cdot 12^{2/m} \cdot 8^{2m^2 + 14m + 24}.
\end{equation}
This concludes the proof of Theorem \ref{concentration}.
\end{proof}

\begin{proof}[Proof of Theorem \ref{mainth}]
Thanks to Proposition \ref{fundprop}, Theorem \ref{mainth} can be immediately derived as a  consequence of  Theorem \ref{concentration} with $L = \frac{i(M) \text{Vol}(\mathbb{S}^m)}{2}$.
This gives
\begin{equation}\label{eqn:littlec}
c(m) = C(m) \cdot \left( \frac{\text{Vol}(\mathbb{S}^m)}{2}\right)^{2/m}.
\end{equation}
\end{proof}

Combining Theorem \ref{mainth} with a result of Milnor \cite{M} gives
\begin{cor}\label{coro1} Let $P_1,\ldots, P_p$ be $p$ real polynomials in $m+p$ variables of degrees $N_1, \ldots, N_p$, respectively, such that  $M=P_1^{-1}(0)\cap\cdots\cap P_p^{-1}(0)\subset \R^{m+p}$ is a compact $m$-dimensional submanifold. Then, for all $k\ge 1$, 
\[
\kthfunchigh \leq   c(m) N_1^{2/m}\cdots N_p^{2/m} k^{2/m} .
\]
\end{cor}
\begin{proof}[Proof of Corollary \ref{coro1}]
 A $p$-plane $\Pi$ in $\R^{m+p}$ is defined as the common zero set of $m$ linearly independent polynomials $P_{p+1}, \ldots, P_{p+m}$, each of degree 1. The zero-dimensional variety $\Pi\cap M$ is then given by the $m +p$ polynomial equations $P_1=0,\cdots, P_{p+m}=0$. According to \cite[Lemma 1]{M}, the
number of points 
in $\Pi\cap M$ is at most equal to the product $(\deg P_1)\cdots (\deg P_{p+m}) = N_1 N_2\cdots N_p$, which implies 
$$i(M)\le N_1 N_2\cdots N_p.$$ 
Applying Theorem \ref{mainth}, we get the result.
\end{proof}

\section{Submanifolds of unit volume and large $\lambda_1$}

The proof of Theorem \ref{large} relies on the following $C^1$ isometric embedding result due to Kuiper \cite{Ku}: If a compact $m$-dimensional smooth manifold $M$ admits a $C^1$ embedding as a submanifold of $\R^{m+p}$, $p\geq 1$, then, given any Riemannian metric $g$ on $M$, there exists a $C^1$ isometric embedding from $(M,g)$ into $\R^{m+p}$. 
\begin{proof}[Proof of  Theorem \ref{large}]
Let $M$ be a compact smooth submanifold of dimension $m\ge 3$ of $\R^{m+p}$, $p\geq 1$, and let $K$ be any positive number. According to the result by Dodziuk and the first author \cite{CD}, there exists a Riemannian metric $g$ on $M$ with $$\lambda_1(g)\text{Vol}(g)^{2/m} \ge 2K.$$ Applying  Kuiper's result cited above, there exists a $C^1$ isometric embedding $Y$ from $(M,g)$ into $\R^{m+p}$. According to standard density theorems (e.g., \cite[p. 50]{Hirsch}), the map $Y$ can be approximated, with arbitrary accuracy with respect to the $C^1$-topology, by a smooth embedding $X$. The smooth  metric $g_1$ induced by $X$ is then quasi-isometric to $g$ with a quasi-isometry ratio arbitrarily close to 1. Consequently, there exists a smooth embedding $X$ satisfying
$$\lambda_1(X(M)) \text{Vol}(X(M))^{2/m} =\lambda_1(g_1)\text{Vol}(g_1)^{2/m} \ge K.$$  
This proves assertion (2).
 
 To prove assertion (1) we use a similar argument.  The result of Colbois-Dodziuk is not valid in dimension 2, so we instead apply a result due to Buser et al. \cite{BBD}: they use an arithmetic construction involving subgroups of the modular group to obtain noncompact hyperbolic surfaces.  After a suitable compactification, these are compact orientable hyperbolic surfaces with large genus and with large first eigenvalue (see also \cite[Thm. C]{CE}).  
 \end{proof} 


\bibliographystyle{plain}
\bibliography{abcde}

\end{document}